\documentclass[review,11pt]{elsarticle}
\usepackage{amssymb,multirow,mathrsfs,subfigure}
\usepackage{amsmath}
\usepackage{array}
\usepackage{algorithm}
\usepackage{algorithmicx}
\usepackage[noend]{algpseudocode}
\usepackage{paralist}
\usepackage{natbib}
\usepackage{hyperref}
\usepackage{epsfig}
\usepackage{color,soul,cancel}
\numberwithin{equation}{section}
\usepackage{mathtools,lipsum}
\usepackage[toc,page]{}
\graphicspath{{images/}{../images/}}
\usepackage{bigints,amsthm}
\allowdisplaybreaks
\newtheorem{theorem}{Theorem}[section]

\makeatletter
\def\ps@pprintTitle{%
\let\@oddhead\@empty
\let\@evenhead\@empty
\let\@oddfoot\@empty
\let\@evenfoot\@oddfoot
}
\makeatother

\begin{document}
\begin{frontmatter}
\title{Convergence analysis of fixed stress split iterative scheme for small strain anisotropic poroelastoplasticity: a primer}
\author[csm]{Saumik Dana}
\ead{saumik@utexas.edu}
\author[csm]{Mary. F. Wheeler}
\ead{mfw@ices.utexas.edu}

\address[csm]{Center for Subsurface Modeling, Institute for Computational Engineering and Sciences, UT Austin, Austin, TX 78712}
\begin{abstract}
This work serves as a primer to our efforts in arriving at convergence estimates for the fixed stress split iterative scheme for single phase flow coupled with small strain anisotropic poroelastoplasticity. The fixed stress split iterative scheme solves the flow subproblem with stress tensor fixed using a mixed finite element method, followed by the poromechanics subproblem using a conforming Galerkin method in every coupling iteration at each time step. The coupling iterations are repeated until convergence and Backward Euler is employed for time marching. The convergence analysis is based on studying the equations satisfied by the difference of iterates to show that the iterative scheme is contractive. 
\end{abstract}
\end{frontmatter}
\section{Introduction}
%
This work follows up on our previous work (\citet{danacg}), where we arrived at convergence estimates for the case of anisotropic poroelasticity with tensor Biot parameter.
\subsection{Preliminaries}
Given a bounded convex domain $\Omega\subset \mathbb{R}^3$, we use $P_k(\Omega)$ to represent the restriction of the space of polynomials of degree less that or equal to $k$ to $\Omega$ and $Q_1(\Omega)$ to denote the space of trilinears on $\Omega$. For the sake of convenience, we discard the differential in the integration of any scalar field $\chi$ over $\Omega$ as follows
\begin{align*}
\int\limits_{\Omega}\chi(\mathbf{x}) \equiv \int\limits_{\Omega}\chi(\mathbf{x}) \,dV\qquad (\forall \mathbf{x}\in \Omega)
\end{align*}
Sobolev spaces are based on the space of square integrable functions on $\Omega$ given by
\begin{align*}
L^2(\Omega)\equiv \big\{\theta:\Vert \theta\Vert_{\Omega}^2:=\int\limits_{\Omega}\vert\theta\vert^2 < +\infty \big\},
\end{align*}
The inner product of two second order tensors $\mathbf{S}$ and $\mathbf{T}$ is given by (see \citet{gurtin})
\begin{align*}
\tag{$i,j=1,2,3$}
\mathbf{S}:\mathbf{T}=S_{ij}T_{ij} 
\end{align*}
A fourth order tensor is a linear transformation of a second order tensor to a second order tensor in the following manner (see \citet{gurtin})
\begin{align*}
\tag{$i,j,k,l=1,2,3$}
\mathbb{P}\mathbf{S}=\mathbf{T}\rightarrow \mathbb{P}_{ijkl}S_{kl}=T_{ij}
\end{align*}
The dyadic product $\otimes$ of two second order tensors $\mathbf{S}$ and $\mathbf{T}$ is given by (see \citet{gurtin})
\begin{align*}
\mathbb{P}=\mathbf{S}\otimes \mathbf{T}\rightarrow \mathbb{P}_{ijkl}=S_{ij}T_{kl}
\end{align*} 
\section{Flow model}\label{flowmodel1} 
Let the boundary $\partial \Omega=\Gamma_D^f \cup \Gamma_N^f$ where $\Gamma_D^f$ is the Dirichlet boundary and $\Gamma_N^f$ is the Neumann boundary. The fluid mass conservation equation \eqref{massagain1} in the presence of deformable and anisotropic porous medium with the Darcy law \eqref{darcyagain1} and linear pressure dependence of density \eqref{compressible1} with boundary conditions \eqref{bcagain11} and initial conditions \eqref{flowendagain1} is
\begin{align}
\label{massagain1}
&\frac{\partial \zeta}{\partial t}+\nabla \cdot \mathbf{z}=q\\
\label{darcyagain1}
&\mathbf{z}=-\frac{\mathbf{K}}{\mu}(\nabla p-\rho_0 \mathbf{g})=-\boldsymbol{\kappa}(\nabla p-\rho_0 \mathbf{g})\\
\label{compressible1}
&\rho=\rho_0(1+c\,(p-p_0))\\
\label{bcagain11}
&p=g \,\, \mathrm{on}\,\,\Gamma_D^f \times (0,T],\,\,\mathbf{z}\cdot\mathbf{n}=0 \,\, \mathrm{on}\,\,\Gamma_N^f \times (0,T]
\\
\nonumber
&p(\mathbf{x},0)=p_0(\mathbf{x}),\,\,\rho(\mathbf{x},0)=\rho_0(\mathbf{x}),\,\, \phi(\mathbf{x},0)=\phi_0(\mathbf{x})\\
\label{flowendagain1}
&(\forall \mathbf{x}\in \Omega)
\end{align}
where $p:\Omega \times (0,T]\rightarrow \mathbb{R}$ is the fluid pressure, $\mathbf{z}:\Omega \times (0,T]\rightarrow \mathbb{R}^3$ is the fluid flux, $\zeta$ is the increment in fluid content\footnote{\citet{biot-energy} defines the increment in fluid content as the measure of the amount of fluid which has flowed in and out of a given element attached to the solid frame}, $\mathbf{n}$ is the unit outward normal on $\Gamma_N^f$, $q$ is the source or sink term, $\mathbf{K}$ is the uniformly symmetric positive definite absolute permeability tensor, $\mu$ is the fluid viscosity, $\rho_0$ is a reference density, $\phi$ is the porosity, $\boldsymbol{\kappa}=\frac{\mathbf{K}}{\mu}$ is a measure of the hydraulic conductivity of the pore fluid, $c$ is the fluid compressibility and $T>0$ is the time interval. 
\section{Poromechanics model}\label{poromodel1}
One of the chief hypotheses underlying the small strain theory of elastoplasticity is the decomposition
of the total strain, $\boldsymbol{\epsilon}$, into the sum of an elastic (or reversible) component $\boldsymbol{\epsilon}^e$, and a plastic (or permanent) component, $\boldsymbol{\epsilon}^p$,
\begin{align*}
\boldsymbol{\epsilon}=\boldsymbol{\epsilon}^e+\boldsymbol{\epsilon}^p
\end{align*}
\par 
Let the boundary $\partial \Omega=\Gamma_D^p \cup \Gamma_N^p$ where $\Gamma_D^p$ is the Dirichlet boundary and $\Gamma_N^p$ is the Neumann boundary. Linear momentum balance for the anisotropic porous solid in the quasi-static limit of interest \eqref{mechstart} with small strain assumption \eqref{smallstrain} with boundary conditions \eqref{bc1} and initial condition \eqref{ic1} is
\begin{align}
\label{mechstart}
&\nabla\cdot \boldsymbol{\sigma}+\mathbf{f}=\mathbf{0}\\
\label{bforce}
&\mathbf{f}=\rho \phi\mathbf{g} + 
\rho_r(1-\phi)\mathbf{g}\\
\label{smallstrain}
&\boldsymbol{\epsilon}(\mathbf{u})=\frac{1}{2}(\nabla \mathbf{u} + (\nabla \mathbf{u})^T)\\
\label{bc1}
&\mathbf{u}\cdot\mathbf{n}_1=0\,\, \mathrm{on}\,\,\Gamma_D^p \times [0,T],\,\,
\boldsymbol{\sigma}^T\mathbf{n}_2=\mathbf{t}\,\, \mathrm{on}\,\,\Gamma_N^p \times [0,T]\\
\label{ic1}
&\mathbf{u}(\mathbf{x},0)=\mathbf{0}\qquad \forall \mathbf{x}\in \Omega
\end{align}
where $\mathbf{u}:\Omega \times [0,T]\rightarrow \mathbb{R}^3$ is the solid displacement, $\rho_r$ is the rock density, $\mathbf{f}$ is the body force per unit volume, $\mathbf{n}_1$ is the unit outward normal to $\Gamma_D^p$,  $\mathbf{n}_2$ is the unit outward normal to $\Gamma_N^p$, $\mathbf{t}$ is the traction specified on $\Gamma_N^p$, $\boldsymbol{\epsilon}$ is the strain tensor, $\boldsymbol{\sigma}$ is the Cauchy stress tensor given by the generalized Hooke's law
\begin{align}
\label{constitutive}
\boldsymbol{\sigma}=\mathbb{D}\boldsymbol{\epsilon}^e
-\boldsymbol{\alpha} p\equiv \mathbb{D}^{ep}\boldsymbol{\epsilon}
-\boldsymbol{\alpha} p
\end{align}
where $\mathbb{D}$ is the fourth order anisotropic elasticity tensor, $\boldsymbol{\alpha}$ is the Biot tensor and $\mathbb{D}^{ep}$ is the elastoplastic tangent operator (see \citet{neto}) 
and $\hat{\boldsymbol{\sigma}}=\mathbb{D} \boldsymbol{\epsilon}^e\equiv \mathbb{D}^{ep} \boldsymbol{\epsilon}$ is the effective stress. The inverse of the generalized Hooke's law \eqref{constitutive} is given by
\begin{align}
\label{invconstitutive}
&\boldsymbol{\epsilon}=\mathbb{D}^{ep^{-1}}(\boldsymbol{\sigma}+\boldsymbol{\alpha}p)=\mathbb{D}^{ep^{-1}}\boldsymbol{\sigma}+\frac{C}{3}\mathbf{B}p
\end{align}
where $C(>0)$ is a generalized Hooke's law constant and $\mathbf{B}$ is a generalization of the Skempton pore pressure coefficient $B$ (see \citet{skempton-1954}) for anisotropic poroelastoplasticity, and is given by
\begin{align}
\label{skempton}
\mathbf{B}\equiv \frac{3}{C}\mathbb{D}^{ep^{-1}}\boldsymbol{\alpha}
\end{align}
\subsection{Increment in fluid content}
The increment in fluid content $\zeta$ is given by (see \citet{coussy})
\begin{align}
\label{fluidcontent}
\zeta=\frac{1}{M}p+\boldsymbol{\alpha}:\boldsymbol{\epsilon}^e+\phi^p \equiv Cp+\frac{1}{3}C\mathbf{B}:\boldsymbol{\sigma}+\phi^p
\end{align}
where $M(>0)$ is a generalization of the Biot modulus (see \citet{biot3}) for anisotropic poroelasticity and $\phi^p$ is a plastic porosity (see \citet{coussy}).
%
\section{Statement of contraction of the fixed stress split scheme for small strain anisotropic poroelastoplasticity with Biot tensor}
We use the notations $(\cdot)^{n+1}$ for any quantity $(\cdot)$ evaluated at time level $n+1$, $(\cdot)^{m,n+1}$ for any quantity $(\cdot)$ evaluated at the $m^{th}$ coupling iteration at time level $n+1$, $\delta^{(m)}_f (\cdot)$ for the change in the quantity $(\cdot)$ during the flow solve over the $(m+1)^{th}$ coupling iteration at any time level and $\delta^{(m)} (\cdot)$ for the change in the quantity $(\cdot)$ over the $(m+1)^{th}$ coupling iteration at any time level. Let $\mathscr{T}_h$ be finite element partition of $\Omega$ consisting of distorted hexahedral elements $E$ where $h=\max\limits_{E\in\mathscr{T}_h}diam(E)$. The details of the finite element mapping are given in \citet{dana-2018}. 
\subsection{Discrete variational statements for the flow subproblem in terms of coupling iteration differences}\label{discreteflow}
Before arriving at the discrete variational statement of the flow model, we impose the fixed stress constraint on the strong form of the mass conservation equation \eqref{massagain1}. In lieu of \eqref{fluidcontent}, we write \eqref{massagain1} as
\begin{align}
\nonumber
&\frac{\partial}{\partial t}(Cp+\frac{C}{3}\mathbf{B}:\boldsymbol{\sigma}+\phi^p)+\nabla \cdot \mathbf{z}=q\\
\label{ek}
&C\frac{\partial p}{\partial t} +\nabla \cdot \mathbf{z}=q-\frac{C}{3}\mathbf{B}:\frac{\partial \boldsymbol{\sigma}}{\partial t}-\frac{\partial \phi^p}{\partial t}
\end{align}
Using backward Euler in time, the discrete in time form of \eqref{ek} for the $m^{th}$ coupling iteration in the $(n+1)^{th}$ time step is written as
\begin{align*}
&C\frac{1}{\Delta t}(p^{m,n+1}-p^n) +\nabla \cdot \mathbf{z}^{m,n+1}\\
&=q^{n+1}-\frac{1}{\Delta t}\frac{C}{3}\mathbf{B}:(\boldsymbol{\sigma}^{m,n+1}-\boldsymbol{\sigma}^n)-\frac{1}{\Delta t}(\phi^{p^{m,n+1}}-\phi^{p^n})
\end{align*}
where $\Delta t$ is the time step and the source term as well as the terms evaluated at the previous time level $n$ do not depend on the coupling iteration count as they are known quantities. The fixed stress constraint implies that $\boldsymbol{\sigma}^{m,n+1}$ gets replaced by $\boldsymbol{\sigma}^{m-1,n+1}$ i.e. the computation of $p^{m,n+1}$ and $\mathbf{z}^{m,n+1}$ is based on the value of stress updated after the poromechanics solve from the previous coupling iteration $m-1$ at the current time level $n+1$. The modified equation is written as
\begin{align}
\nonumber
&C(p^{m,n+1}-p^n)+\Delta t\nabla \cdot \mathbf{z}^{m,n+1}=\Delta t q^{n+1}-\frac{C}{3}\mathbf{B}:(\boldsymbol{\sigma}^{m,n+1}-\boldsymbol{\sigma}^n)-(\phi^{p^{m,n+1}}-\phi^{p^n})
\end{align}
As a result, the discrete weak form of \eqref{massagain1} is given by
\begin{align}
\nonumber
&C(p_h^{m,n+1}-p_h^n,\theta_h)_{\Omega}+\Delta t(\nabla \cdot \mathbf{z}_h^{m,n+1},\theta_h)_{\Omega}+(\phi^{p^{m,n+1}}-\phi^{p^n},\theta_h)_{\Omega}\\
\nonumber
&=\Delta t(q^{n+1},\theta_h)_{\Omega}-\frac{C}{3}(\mathbf{B}:(\boldsymbol{\sigma}^{m-1,n+1}-\boldsymbol{\sigma}^n),\theta_h)_{\Omega}
\end{align}
Replacing $m$ by $m+1$ and subtracting the two equations, we get 
\begin{align*}
&C(\delta^{(m)}p_h,\theta_h)_{\Omega}+\Delta t(\nabla \cdot \delta^{(m)}\mathbf{z}_h,\theta_h)_{\Omega}+(\delta^{(m)}\phi^p,\theta_h)_{\Omega}=-\frac{C}{3}(\mathbf{B}:\delta^{(m-1)}\boldsymbol{\sigma},\theta_h)_{\Omega}
\end{align*}
The weak form of the Darcy law \eqref{darcyagain1} for the $m^{th}$ coupling iteration in the $(n+1)^{th}$ time step is given by
\begin{align}
\label{wtwo1}
(\boldsymbol{\kappa}^{-1}\mathbf{z}^{m,n+1},\mathbf{v})_{\Omega}=-(\nabla p^{m,n+1},\mathbf{v})_{\Omega}+(\rho_0 \mathbf{g},\mathbf{v})_{\Omega}\,\, \forall\,\,\mathbf{v}\in \mathbf{V}(\Omega)
\end{align}
where $\mathbf{V}(\Omega)$ is given by
\begin{align*}
\mathbf{V}(\Omega)\equiv \mathbf{H}(div,\Omega)\cap \big\{\mathbf{v}:\mathbf{v}\cdot \mathbf{n}=0\,\,\mathrm{on}\,\,\Gamma_N^f\big\}
\end{align*}
and $\mathbf{H}(div,\Omega)$ is given by 
\begin{align*}
\mathbf{H}(div,\Omega)\equiv\big\{\mathbf{v}:\mathbf{v}\in (L^2(\Omega))^3,\nabla \cdot \mathbf{v}\in L^2(\Omega) \big\}
\end{align*}
We use the divergence theorem to evaluate the first term on RHS of \eqref{wtwo1} as follows
\begin{align}
\nonumber
&(\nabla p^{m,n+1},\mathbf{v})_{\Omega}=(\nabla,p^{m,n+1}\mathbf{v})_{\Omega}-(p^{m,n+1},\nabla \cdot \mathbf{v})_{\Omega}\\
\nonumber
&=(p^{m,n+1},\mathbf{v}\cdot \mathbf{n})_{\partial \Omega} -(p^{m,n+1},\nabla \cdot \mathbf{v})_{\Omega}\\
\label{wtwo2}
&=(g,\mathbf{v}\cdot \mathbf{n})_{\Gamma_D^f} -(p^{m,n+1},\nabla \cdot \mathbf{v})_{\Omega}
\end{align}
where we invoke $\mathbf{v}\cdot \mathbf{n}=0$ on $\Gamma_N^f$.
In lieu of \eqref{wtwo1} and \eqref{wtwo2}, we get
\begin{align*}
(\boldsymbol{\kappa}^{-1}\mathbf{z}^{m,n+1},\mathbf{v})_{\Omega}=-(g,\mathbf{v}\cdot \mathbf{n})_{\Gamma_D^f}+(p^{m,n+1},\nabla \cdot \mathbf{v})_{\Omega}+(\rho_0 \mathbf{g},\mathbf{v})_{\Omega}
\end{align*}
Replacing $m$ by $m+1$ and subtracting the two equations, we get
\begin{align}
\nonumber
&(\boldsymbol{\kappa}^{-1}\delta^{(m)} \mathbf{z}_h, \mathbf{v}_h)_{\Omega}=(\delta^{(m)} p_h,\nabla \cdot \mathbf{v}_h)_{\Omega}
\end{align} 
\subsection{Discrete variational statement for the poromechanics subproblem in terms of coupling iteration differences}\label{discretemechanics}
The weak form of the linear momentum balance \eqref{mechstart} is given by
\begin{align}
\label{app1}
(\nabla \cdot \boldsymbol{\sigma},\mathbf{q})_{\Omega}+(\mathbf{f}\cdot \mathbf{q})_{\Omega}=0\qquad (\forall\,\,\mathbf{q}\in \mathbf{U}(\Omega))
\end{align}
where $\mathbf{U}(\Omega)$ is given by
\begin{align*}
\mathbf{U}(\Omega)\equiv \big\{\mathbf{q}=(u,v,w):u,v,w\in H^1(\Omega),\mathbf{q}=\mathbf{0}\,\,\mathrm{on}\,\,\Gamma_D^p\big\}
\end{align*}
where $H^m(\Omega)$ is defined, in general, for any integer $m\geq 0$ as
\begin{align*}
H^m(\Omega)\equiv\big\{w:D^{\alpha}w\in L^2(\Omega)\,\,\forall |\alpha| \leq m \big\},
\end{align*}
where the derivatives are taken in the sense of distributions and given by
\begin{align*}
D^{\alpha}w=\frac{\partial^{|\alpha|}w}{\partial x_1^{\alpha_1}..\partial x_n^{\alpha_n}},\,\,|\alpha|=\alpha_1+\cdots+\alpha_n,
\end{align*} 
We know from tensor calculus that
\begin{align}
\label{app2}
(\nabla \cdot \boldsymbol{\sigma},\mathbf{q})_{\Omega}\equiv (\nabla ,\boldsymbol{\sigma}\mathbf{q})_{\Omega}-(\boldsymbol{\sigma}:\nabla \mathbf{q})_{\Omega}
\end{align}
Further, using the divergence theorem and the symmetry of $\boldsymbol{\sigma}$, we arrive at
\begin{align}
\label{app3}
(\nabla ,\boldsymbol{\sigma}\mathbf{q})_{\Omega}\equiv (\mathbf{q},\boldsymbol{\sigma}\mathbf{n})_{\partial \Omega}
\end{align}
We decompose $\nabla \mathbf{q}$ into a symmetric part $(\nabla \mathbf{q})_{s}\equiv \frac{1}{2}\big(\nabla \mathbf{q}+(\nabla \mathbf{q})^T\big)\equiv \boldsymbol{\epsilon}(\mathbf{q})$ and skew-symmetric part $(\nabla \mathbf{q})_{ss}$ and note that the contraction between a symmetric and skew-symmetric tensor is zero to obtain
From \eqref{app1}, \eqref{app2}, \eqref{app3} and \eqref{app4}, we get
\begin{align*}
&(\boldsymbol{\sigma}\mathbf{n},\mathbf{q})_{\partial \Omega} - (\boldsymbol{\sigma}:\boldsymbol{\epsilon}(\mathbf{q}))_{\Omega} + (\mathbf{f},\mathbf{q})_{\Omega}=0
\end{align*}
which, after invoking the traction boundary condition, results in the discrete weak form for the $m^{th}$ coupling iteration as
\begin{align}
\nonumber
&(\mathbf{t}^{n+1},\mathbf{q}_h)_{\Gamma_N^p} - (\boldsymbol{\sigma}^{m,n+1}:\boldsymbol{\epsilon}(\mathbf{q}_h))_{\Omega} + (\mathbf{f}^{n+1},\mathbf{q}_h)_{\Omega}=0
\end{align}
Replacing $m$ by $m+1$ and subtracting the two equations, we get
\begin{align}
\nonumber
&(\delta^{(m)} \boldsymbol{\sigma}:\boldsymbol{\epsilon}(\mathbf{q}_h))_{\Omega}=0
\end{align}
\subsection{Summary}
The discrete variational statements in terms of coupling iteration differences is : find $\delta^{(m)} p_h\in W_h$, $\delta^{(m)} \mathbf{z}_h\in \mathbf{V}_h$ and $\delta^{(m)} \mathbf{u}_h\in \mathbf{U}_h$ such that
\begin{align}
\label{msone}
&C(\delta^{(m)}p_h,\theta_h)_{\Omega}+\Delta t(\nabla \cdot \delta^{(m)}\mathbf{z}_h,\theta_h)_{\Omega}+(\delta^{(m)}\phi^p,\theta_h)_{\Omega}=-\frac{C}{3}(\mathbf{B}:\delta^{(m-1)}\boldsymbol{\sigma},\theta_h)_{\Omega}\\
\label{mstwo}
&(\boldsymbol{\kappa}^{-1}\delta^{(m)} \mathbf{z}_h, \mathbf{v}_h)_{\Omega}=(\delta^{(m)} p_h,\nabla \cdot \mathbf{v}_h)_{\Omega}\\
\label{msthree}
&(\delta^{(m)} \boldsymbol{\sigma}:\boldsymbol{\epsilon}(\mathbf{q}_h))_{\Omega}=0
\end{align}
where the finite dimensional spaces $W_h$, $\mathbf{V}_h$ and $\mathbf{U}_h$ are
\begin{align*}
&W_h= \big\{\theta_h:\theta_h\vert\in P_0(E)\,\,\forall E\in \mathscr{T}_h\big\}\\
&\mathbf{V}_h=\big\{\mathbf{v}_h:\mathbf{v}_h\vert_E\leftrightarrow \hat{\mathbf{v}}\vert_{\hat{E}}\in \hat{\mathbf{V}}(\hat{E})\,\,\forall E\in \mathscr{T}_h,\,\,\mathbf{v}_h \cdot \mathbf{n}=0\,\,\mathrm{on}\,\,\Gamma_N^f\big\}\\
&\mathbf{U}_h=\big\{\mathbf{q}_h=(u,v,w)\vert_E\in Q_1(E)\,\,\forall E\in \mathscr{T}_h,\,\,\mathbf{q}_h=\mathbf{0}\,\,\mathrm{on}\,\,\Gamma_D^p\big\}
\end{align*}   
where $P_0$ represents the space of constants, $Q_1$ represents the space of trilinears and the details of $\hat{\mathbf{V}}(\hat{E})$ are given in \citet{dana-2018}. 
\begin{theorem}\label{map}
The fixed stress split iterative coupling scheme for anisotropic poroelasticity with Biot tensor in which the flow problem is solved first by freezing all components of the stress tensor is a contraction given by 
\begin{align*}
\nonumber
&\frac{C}{6}\Vert \mathbf{B}:\delta^{(m)}\boldsymbol{\sigma}\Vert_{\Omega}^2+\overbrace{\frac{C}{2}\Vert\delta^{(m)}p_h\Vert_{\Omega}^2}^{>0}+\overbrace{\Delta t
\Vert\boldsymbol{\kappa}^{-1/2}\delta^{(m)} \mathbf{z}_h \Vert^2_{\Omega}}^{>0}\\
\nonumber
&+\overbrace{(\delta^{(m)} \boldsymbol{\sigma}:\mathbb{D}^{ep^{-1}}\delta^{(m)}\boldsymbol{\sigma})_{\Omega}}^{\geq 0?}+\overbrace{\frac{1}{2C}\Vert \delta^{(m)}\zeta\Vert_{\Omega}^2}^{>0}+\overbrace{\frac{1}{C}\Vert \delta^{(m)}\phi^p-\delta^{(m)}_f\phi^p\Vert_{\Omega}^2}^{\geq 0}\\
\nonumber
&-\bigg[\overbrace{\frac{1}{C}\Vert \delta^{(m)}\zeta-\delta^{(m)}_f\zeta\Vert_{\Omega}^2+\frac{1}{2C}\Vert \delta^{(m)}\phi^p\Vert_{\Omega}^2+\frac{1}{3}(\mathbf{B}:\delta^{(m)}\boldsymbol{\sigma},\delta_f^{(m)}\phi^p)_{\Omega}}^{\mathrm{driven\,\,to\,\,zero\,\,by\,\,convergence\,\,criterion}}\bigg]\\
\nonumber
&\leq \frac{C}{6}\Vert \mathbf{B}:\delta^{(m-1)}\boldsymbol{\sigma}\Vert_{\Omega}^2
\end{align*}
\end{theorem}
\begin{proof}
$\bullet$ \textbf{Step 1: Flow equations}\\ 
Testing \eqref{msone} with $\theta_h\equiv \delta^{(m)} p_h$, we get
\begin{align}
\nonumber
&C\Vert\delta^{(m)}p_h\Vert_{\Omega}^2+\Delta t(\nabla \cdot \delta^{(m)}\mathbf{z}_h,\delta^{(m)} p_h)_{\Omega}+(\delta^{(m)}\phi^p,\delta^{(m)} p_h)_{\Omega}\\
\label{msfour}
&=-\frac{C}{3}(\mathbf{B}:\delta^{(m-1)}\boldsymbol{\sigma},\delta^{(m)} p_h)_{\Omega}
\end{align}
Testing \eqref{mstwo} with $\mathbf{v}_h\equiv \delta^{(m)} \mathbf{z}_h$, we get
\begin{align}
\label{msfive}
&\Vert\boldsymbol{\kappa}^{-1/2}\delta^{(m)} \mathbf{z}_h \Vert^2_{\Omega}=(\delta^{(m)} p_h,\nabla \cdot \delta^{(m)} \mathbf{z}_h)_{\Omega}
\end{align}
From \eqref{msfour} and \eqref{msfive}, we get
\begin{align}
\label{mssix}
&C\Vert\delta^{(m)}p_h\Vert_{\Omega}^2+\Delta t
\Vert\boldsymbol{\kappa}^{-1/2}\delta^{(m)} \mathbf{z}_h \Vert^2_{\Omega}+(\delta^{(m)}\phi^p,\delta^{(m)} p_h)_{\Omega}=-\frac{C}{3}(\mathbf{B}:\delta^{(m-1)}\boldsymbol{\sigma},\delta^{(m)} p_h)_{\Omega}
\end{align}
$\bullet$ \textbf{Step 2: Poromechanics equations}\\
Testing \eqref{msthree} with $\mathbf{q}_h\equiv \delta^{(m)}\mathbf{u}_h$, we get
\begin{align}
\label{eq1}
&(\delta^{(m)} \boldsymbol{\sigma}:\delta^{(m)}\boldsymbol{\epsilon})_{\Omega}=0
\end{align}
We now invoke \eqref{invconstitutive} to arrive at the expression for change in strain tensor over the $(m+1)^{th}$ coupling iteration as follows
\begin{align}
\label{eq2}
\delta^{(m)}\boldsymbol{\epsilon}=\mathbb{D}^{ep^{-1}}\delta^{(m)}\boldsymbol{\sigma}+\frac{C}{3}\mathbf{B}\delta^{(m)}p_h
\end{align}
Substituting \eqref{eq2} in \eqref{eq1}, we get
\begin{align}
\label{msseven}
&(\delta^{(m)} \boldsymbol{\sigma}:\mathbb{D}^{ep^{-1}}\delta^{(m)}\boldsymbol{\sigma})_{\Omega}+\frac{C}{3}(\mathbf{B}:\delta^{(m)} \boldsymbol{\sigma},\delta^{(m)}p_h)_{\Omega}=0
\end{align}
$\bullet$ \textbf{Step 3: Combining flow and poromechanics equations}\\
Adding \eqref{mssix} and \eqref{msseven}, we get
\begin{align}
\nonumber
&C\Vert\delta^{(m)}p_h\Vert_{\Omega}^2+\Delta t
\Vert\boldsymbol{\kappa}^{-1/2}\delta^{(m)} \mathbf{z}_h \Vert^2_{\Omega}+(\delta^{(m)}\phi^p,\delta^{(m)} p_h)_{\Omega}+(\delta^{(m)} \boldsymbol{\sigma}:\mathbb{D}^{ep^{-1}}\delta^{(m)}\boldsymbol{\sigma})_{\Omega}\\
\label{use1}
&+\frac{C}{3}(\mathbf{B}:\delta^{(m)} \boldsymbol{\sigma},\delta^{(m)}p_h)_{\Omega}=-\frac{C}{3}(\mathbf{B}:\delta^{(m-1)}\boldsymbol{\sigma},\delta^{(m)} p_h)_{\Omega}
\end{align}
$\bullet$ \textbf{Step 4: Variation in fluid content}\\
In lieu of \eqref{fluidcontent}, the variation in fluid content in the $(m+1)^{th}$ coupling iteration is
\begin{align}
\label{fluidcontentfull}
&\delta^{(m)}\zeta=C\delta^{(m)}p_h+\frac{C}{3}\mathbf{B}:\delta^{(m)}\boldsymbol{\sigma}+\delta^{(m)}\phi^p
\end{align}
As a result, we can write
\begin{align}
\nonumber
&\frac{1}{2C}\Vert \delta^{(m)}\zeta\Vert_{\Omega}^2-\frac{C}{2}\Vert\delta^{(m)}p_h\Vert_{\Omega}^2-\frac{C}{18}\Vert \mathbf{B}:\delta^{(m)}\boldsymbol{\sigma}\Vert_{\Omega}^2-\frac{1}{2C}\Vert \delta^{(m)}\phi^p\Vert_{\Omega}^2\\
\label{daal}
&-(\delta^{(m)}\phi^p,\delta^{(m)} p_h)_{\Omega}-\frac{1}{3}(\mathbf{B}:\delta^{(m)}\boldsymbol{\sigma},\delta^{(m)}\phi^p)_{\Omega}=\frac{C}{3}(\mathbf{B}:\delta^{(m)} \boldsymbol{\sigma},\delta^{(m)}p_h)_{\Omega}
\end{align}
From \eqref{use1} and \eqref{daal}, we get
\begin{align}
\nonumber
&C\Vert\delta^{(m)}p_h\Vert_{\Omega}^2+\Delta t
\Vert\boldsymbol{\kappa}^{-1/2}\delta^{(m)} \mathbf{z}_h \Vert^2_{\Omega}+(\delta^{(m)} \boldsymbol{\sigma}:\mathbb{D}^{ep^{-1}}\delta^{(m)}\boldsymbol{\sigma})_{\Omega}+\frac{1}{2C}\Vert \delta^{(m)}\zeta\Vert_{\Omega}^2\\
\nonumber
&-\frac{C}{2}\Vert\delta^{(m)}p_h\Vert_{\Omega}^2-\frac{C}{18}\Vert \mathbf{B}:\delta^{(m)}\boldsymbol{\sigma}\Vert_{\Omega}^2-\frac{1}{2C}\Vert \delta^{(m)}\phi^p\Vert_{\Omega}^2-\frac{1}{3}(\mathbf{B}:\delta^{(m)}\boldsymbol{\sigma},\delta^{(m)}\phi^p)_{\Omega}\\
\label{use2}
&=-\frac{C}{3}(\mathbf{B}:\delta^{(m-1)}\boldsymbol{\sigma},\delta^{(m)} p_h)_{\Omega}
\end{align}
Adding and subtracting $\frac{C}{6}\Vert \mathbf{B}:\delta^{(m)}\boldsymbol{\sigma}\Vert_{\Omega}^2$ to the LHS of \eqref{use2} results in
\begin{align}
\nonumber
&\frac{C}{6}\Vert \mathbf{B}:\delta^{(m)}\boldsymbol{\sigma}\Vert_{\Omega}^2+\frac{C}{2}\Vert\delta^{(m)}p_h\Vert_{\Omega}^2+\Delta t
\Vert\boldsymbol{\kappa}^{-1/2}\delta^{(m)} \mathbf{z}_h \Vert^2_{\Omega}+(\delta^{(m)} \boldsymbol{\sigma}:\mathbb{D}^{ep^{-1}}\delta^{(m)}\boldsymbol{\sigma})_{\Omega}\\
\nonumber
&+\frac{1}{2C}\Vert \delta^{(m)}\zeta\Vert_{\Omega}^2-\frac{C}{9}\Vert \mathbf{B}:\delta^{(m)}\boldsymbol{\sigma}\Vert_{\Omega}^2-\frac{1}{2C}\Vert \delta^{(m)}\phi^p\Vert_{\Omega}^2-\frac{1}{3}(\mathbf{B}:\delta^{(m)}\boldsymbol{\sigma},\delta^{(m)}\phi^p)_{\Omega}\\
\label{use3}
&=-\frac{C}{3}(\mathbf{B}:\delta^{(m-1)}\boldsymbol{\sigma},\delta^{(m)} p_h)_{\Omega}
\end{align}
In lieu of \eqref{fluidcontent} and the fixed stress constraint during the flow solve, the variation in fluid content during the flow solve in the $(m+1)^{th}$ coupling iteration is given by
\begin{align*}
&\delta^{(m)}_f\zeta=C\delta_f^{(m)}p_h+\frac{C}{3}\mathbf{B}:\cancelto{\mathbf{0}}{\delta^{(m)}_f\boldsymbol{\sigma}}+\delta_f^{(m)}\phi^p
\end{align*}
Further, since the pore pressure is frozen during the poromechanical solve, we have $\delta^{(m)}_fp_h=\delta^{(m)}p_h$. As a result, we can write
\begin{align}
\label{fluidcontentflow}
&\delta^{(m)}_f\zeta=C\delta^{(m)}p_h+\delta_f^{(m)}\phi^p
\end{align}
Subtracting \eqref{fluidcontentflow} from \eqref{fluidcontentfull}, we can write
\begin{align}
\nonumber
&\delta^{(m)}\zeta-\delta^{(m)}_f\zeta=\frac{C}{3}\mathbf{B}:\delta^{(m)}\boldsymbol{\sigma}+\delta^{(m)}\phi^p-\delta^{(m)}_f\phi^p
\end{align}
which implies that
\begin{align}
\nonumber
&\frac{1}{C}\Vert \delta^{(m)}\zeta-\delta^{(m)}_f\zeta\Vert_{\Omega}^2-\frac{1}{C}\Vert \delta^{(m)}\phi^p-\delta^{(m)}_f\phi^p\Vert_{\Omega}^2\\
\label{convergence}
&-\frac{1}{3}(\mathbf{B}:\delta^{(m)}\boldsymbol{\sigma},(\delta^{(m)}\phi^p-\delta_f^{(m)}\phi^p))_{\Omega}=\frac{C}{9}\Vert \mathbf{B}:\delta^{(m)}\boldsymbol{\sigma}\Vert_{\Omega}^2
\end{align}
In lieu of \eqref{convergence}, we can write \eqref{use3} as
\begin{align}
\nonumber
&\frac{C}{6}\Vert \mathbf{B}:\delta^{(m)}\boldsymbol{\sigma}\Vert_{\Omega}^2+\overbrace{\frac{C}{2}\Vert\delta^{(m)}p_h\Vert_{\Omega}^2}^{\geq 0}+\overbrace{\Delta t
\Vert\boldsymbol{\kappa}^{-1/2}\delta^{(m)} \mathbf{z}_h \Vert^2_{\Omega}}^{\geq 0}\\
\nonumber
&+\overbrace{(\delta^{(m)} \boldsymbol{\sigma}:\mathbb{D}^{ep^{-1}}\delta^{(m)}\boldsymbol{\sigma})_{\Omega}}^{\geq 0?}+\overbrace{\frac{1}{2C}\Vert \delta^{(m)}\zeta\Vert_{\Omega}^2}^{\geq 0}+\overbrace{\frac{1}{C}\Vert \delta^{(m)}\phi^p-\delta^{(m)}_f\phi^p\Vert_{\Omega}^2}^{\geq 0}\\
\nonumber
&-\bigg[\overbrace{\frac{1}{C}\Vert \delta^{(m)}\zeta-\delta^{(m)}_f\zeta\Vert_{\Omega}^2+\frac{1}{2C}\Vert \delta^{(m)}\phi^p\Vert_{\Omega}^2+\frac{1}{3}(\mathbf{B}:\delta^{(m)}\boldsymbol{\sigma},\delta_f^{(m)}\phi^p)_{\Omega}}^{\mathrm{driven\,\,to\,\,zero\,\,by\,\,convergence\,\,criterion}}\bigg]\\
\label{mseight}
&=-\frac{C}{3}(\mathbf{B}:\delta^{(m-1)}\boldsymbol{\sigma},\delta^{(m)} p_h)_{\Omega}
\end{align}
\\
$\bullet$ \textbf{Step 5: Invoking the Young's inequality}\\
Since the sum of the terms on the LHS of \eqref{mseight} is nonnegative, the RHS is also nonnegative. We invoke the Young's inequality (see \citet{steele})
for the RHS of \eqref{mseight} as follows
\begin{align}
\label{young}
&-\frac{C}{3}(\mathbf{B}:\delta^{(m-1)}\boldsymbol{\sigma},\delta^{(m)} p_h)_{\Omega}\leq \frac{C}{3}\bigg(\frac{1}{2} \Vert \mathbf{B}:\delta^{(m-1)}\boldsymbol{\sigma}\Vert_{\Omega}^2+\frac{1}{2}\Vert\delta^{(m)} p_h \Vert_{\Omega}^2\bigg)
\end{align}
In lieu of \eqref{young}, we write \eqref{mseight} as
\begin{align}
\nonumber
&\frac{C}{6}\Vert \mathbf{B}:\delta^{(m)}\boldsymbol{\sigma}\Vert_{\Omega}^2+\overbrace{\frac{C}{2}\Vert\delta^{(m)}p_h\Vert_{\Omega}^2}^{>0}+\overbrace{\Delta t
\Vert\boldsymbol{\kappa}^{-1/2}\delta^{(m)} \mathbf{z}_h \Vert^2_{\Omega}}^{>0}\\
\nonumber
&+\overbrace{(\delta^{(m)} \boldsymbol{\sigma}:\mathbb{D}^{ep^{-1}}\delta^{(m)}\boldsymbol{\sigma})_{\Omega}}^{\geq 0?}+\overbrace{\frac{1}{2C}\Vert \delta^{(m)}\zeta\Vert_{\Omega}^2}^{>0}+\overbrace{\frac{1}{C}\Vert \delta^{(m)}\phi^p-\delta^{(m)}_f\phi^p\Vert_{\Omega}^2}^{\geq 0}\\
\nonumber
&-\bigg[\overbrace{\frac{1}{C}\Vert \delta^{(m)}\zeta-\delta^{(m)}_f\zeta\Vert_{\Omega}^2+\frac{1}{2C}\Vert \delta^{(m)}\phi^p\Vert_{\Omega}^2+\frac{1}{3}(\mathbf{B}:\delta^{(m)}\boldsymbol{\sigma},\delta_f^{(m)}\phi^p)_{\Omega}}^{\mathrm{driven\,\,to\,\,zero\,\,by\,\,convergence\,\,criterion}}\bigg]\\
\nonumber
&\leq \frac{C}{6}\Vert \mathbf{B}:\delta^{(m-1)}\boldsymbol{\sigma}\Vert_{\Omega}^2
\end{align}
\\
$\bullet$ \textbf{Step 6: On the agenda}
\begin{compactitem}
\item To render a complete statement of contraction, we need to arrive at estimates for the term $(\delta^{(m)} \boldsymbol{\sigma}:\mathbb{D}^{ep^{-1}}\delta^{(m)}\boldsymbol{\sigma})_{\Omega}$. In lieu of that, we need to arrive at the inverse of the elastoplastic tangent operator $\mathbb{D}^{ep}$. The exact expression for $\mathbb{D}^{ep}$ shall depend on the type of plasticity model (associative or non-associative) and the yield criterion being employed. As we know the expression for $\mathbb{D}^{ep}$, we intend to invert it by using the Sherman-Morrison formula (see \citet{jac-1950}). 
\item In addition, the term $\frac{1}{C}\Vert \delta^{(m)}\zeta-\delta^{(m)}_f\zeta\Vert_{\Omega}^2+\frac{1}{2C}\Vert \delta^{(m)}\phi^p\Vert_{\Omega}^2+\frac{1}{3}(\mathbf{B}:\delta^{(m)}\boldsymbol{\sigma},\delta_f^{(m)}\phi^p)_{\Omega}$ needs to be driven to zero to achieve optimal convergence rate. We intend to do that by designing the convergence criterion as that term going to a small positive value. But before we do that, we need to arrive at an expression for the plastic porosity $\phi^p$ by using certain hypotheses given in \citet{coussy}).
\end{compactitem}
\end{proof}
\bibliographystyle{plainnat}
\bibliography{diss}
\end{document}